\documentclass[10pt]{amsart}

\usepackage{amsmath, amssymb, amsbsy}
\usepackage{array}
\usepackage{graphpap, color, paralist,xcolor}
\usepackage[mathscr]{eucal}
\usepackage[pdftex]{graphicx}
\usepackage[pdftex,colorlinks,backref=page,citecolor=blue]{hyperref}
\usepackage{pifont}
\usepackage{tikz}
\tikzset{sdot/.style = {fill, circle, inner sep = 1.25pt}}
\usepackage{mathtools}
\usepackage{cleveref}
\usepackage{float}
\usepackage{crossreftools}

\newcounter{bullet}

\usepackage{setspace}
\setdisplayskipstretch{1}
\newenvironment{subproof}[1][\proofname]{
  
  \begin{proof}[#1]
}{
  \end{proof}
}
\usepackage{geometry}
\usepackage{comment}
\newtheorem{thm}{Theorem}[section]
\newtheorem{prop}[thm]{Proposition}
\newtheorem{cor}[thm]{Corollary}
\newtheorem{lem}[thm]{Lemma}

\theoremstyle{definition}

\newtheorem{claim}[thm]{Claim}
\newtheorem{example}[thm]{Example}

\newtheorem{remark}[thm]{Remark}

\crefname{lem}{lemma}{lemmas}
\newcommand{\gO}{\Omega}

\newcommand{\Z}{\mathbb{Z}}

\newcommand{\cC}{\mathcal{C} }

\newcommand{\cG}{\mathcal{G} }

\newcommand{\bff}{\mathbf{f}}

\newcommand{\beq}[1]{\begin{equation}\label{#1}}
\newcommand{\enq}[0]{\end{equation}}
\newcommand{\eps}{\epsilon}

\newcommand{\gl}[0]{\lambda}

\newcommand{\mn}[0]{\medskip\noindent}
\newcommand{\nin}[0]{\noindent}

\newcommand{\sub}[0]{\subseteq}

\newcommand{\ra}[0]{\rightarrow}

\newcommand{\pr}[0]{\mathbb{P}}

\newcommand{\Hom}[0]{\mbox{\rm{Hom}}}
\newcommand{\Lip}[0]{\mbox{\rm{Lip}}}
\newcommand{\dist}[0]{\mbox{\rm{dist}}}

\begin{document}

\title{Random Lipschitz functions on graphs with weak expansion}

\author[S. I\c{s}\i k]{Senem I\c{s}\i k}
\address{Department of Mathematics, Stanford University}
\email{senemi@stanford.edu}

\author[J. Park]{Jinyoung Park}
\address{Department of Mathematics, Courant Institute of Mathematical Sciences, New York University}
\email{jinyoungpark@nyu.edu}

\begin{abstract}
Benjamini, Yadin, and Yehudayoff (2007) showed that if the maximum degree of a graph $G$ is 'sub-logarithmic,' then the typical range of random $\mathbb Z$-homomorphisms is super-constant. Furthermore, they showed that there is a sharp transition on the range of random $\Z$-homomorphisms on the graph $C_{n,k}$, the tensor product of the $n$-cycle and the complete graph on $k$ vertices with self-loops, around $k=2\log n$. We extend (to some extent) their results to random $M$-Lipschitz functions and random real-valued Lipschitz functions.
\end{abstract}

\maketitle

\section{Introduction}\label{sec.intro}

The present work is most closely motivated by the works of Benjamini, Yadin, and Yehudayoff \cite{benjamini2007random} and Peled, Samotij, and Yehudayoff \cite{peled2013lipschitz}  which study the range of random Lipschitz functions on graphs. Throughout the paper, $G=(V,E)$ is a connected finite graph. An \textit{$M$-Lipschitz function} on $G$ is an integer-valued function $f:V(G) \ra \Z$ such that $|f(u)-f(v)|\le M$ for all $\{u,v\} \in E(G)$, and an \textit{$\mathbb R$-valued Lipschitz function} on $G$ is a function $f:V(G) \ra \mathbb R$ such that $|f(u)-f(v)| \le 1$ for all $\{u,v\} \in E(G)$. The study of random Lipschitz functions on graphs can be seen as a special case of the study of random surfaces in statistical physics; we refer the readers to \cite[Section 2]{peled2013grounded} for more details. Another important class of Lipschitz functions in our context is a \textit{$\Z$-homomorphism} on $G$, a function $f:V(G)\ra \mathbb Z$ such that $|f(u)-f(v)|=1$ for all $\{u,v\} \in E(G)$. There is a rich history of study of random $\Z$-homomorphisms on various graphs, see e.g. \cite{benjamini2000random, benjamini1994tree}.

The \textit{range} of a function $f: V(G) \ra \mathbb R$ is defined to be
\[R(f):=\max_{v \in V(G)}f(v)-\min_{v\in V(G)} f(v)+1.\]
For an arbitrarily chosen $v_0 \in V(G),$ write
\beq{def.M-Lip} \Lip_{v_0}(G;M)=\{f:V(G) \ra \Z \text{ such that $f(v_0)=0$ and $|f(u)-f(v)|\le M$ $ \forall \{u,v\} \in E(G)$}\}\enq
and
\beq{def.Lip} \Lip_{v_0}(G;\infty)=\{f:V(G) \ra \mathbb R \text{ such that $f(v_0)=0$ and $|f(u)-f(v)|\le 1$ $ \forall \{u,v\} \in E(G)$}\},\enq
noting that both $\Lip_{v_0}(G;M)$ and $\Lip_{v_0}(G;\infty)$ are finite.

The main focus of this paper is the distribution of $R(f)$ for a uniformly randomly chosen $f$ from the families in \eqref{def.M-Lip} or \eqref{def.Lip}. Denote by $x \in_R X$ a uniformly random element $x$ of $X$. Observe that the choice of $v_0$ in \eqref{def.M-Lip} or \eqref{def.Lip} doesn't affect the distribution of $R(f)$ for $f \in_R \Lip_{v_0}(G;M)$ or $f \in_R \Lip_{v_0}(G;\infty)$, because for any $v_0, v_0' \in V(G)$, we can define a natural bijection between $\Lip_{v_0}(G;M)$ and $\Lip_{v_0'}(G;M)$ using translations (and similarly for $\Lip_{v_0}(G;\infty)$), which preserves the range.

Our first main result shows that, if the host graph $G$ exhibits a "small expansion," a random Lipschitz function typically exhibits a "large" fluctuation. In the statement below, $B_r(v) ~(\sub V(G))$ is the \textit{ball of radius $r$ centered at $v$}. We use $\log$ for $\log_2$, and for the results in this paper, we don't try to optimize constant factors.

\begin{thm}\label{MT_M-Lips}
There is a constant $c>0$ for which the following holds.    Let $G$ be a connected graph on $n$ vertices. Fix $v_0 \in V(G)$, and let $f$ be chosen uniformly at random from $\Lip_{v_0}(G;M)$. Then for any $r$ that satisfies $\max\{|B_{r-1}(v)|:v \in V(G)\} \le c\log n$,
    \beq{MT_conclusion} \pr\left[R(f)< \frac{Mr}{2}\right] =o_n(1).\enq
\end{thm}
\nin (Throughout the paper, $o_n(1)$ means the quantity tends to 0 as $n \ra \infty$ (and doesn't depend on other parameters), and $o_n(\cdot)$ is used similarly.) The function $o_n(1)$ on the right-side of \eqref{MT_conclusion} doesn't depend on $M$, so in particular, \Cref{MT_M-Lips} holds even for $M \ra \infty$. We can take the constant $c$ slightly smaller than $1/2$ if $M$ is large enough.

The main motivation for \Cref{MT_M-Lips} comes from the work in \cite{peled2013lipschitz} (ths result there was recently improved in \cite{krueger2024lipschitz}), which provides an {upper bound} on the range of random Lipschitz functions (``flatness of Lipschitz functions") when the host graph $G$ exhibits a "large expansion." Here we define the notion of large expansion more rigorously: following \cite{peled2013lipschitz}, say an $n$-vertex, $d$-regular graph $G$ is a \emph{$\lambda$-expander} if
\beq{def.expander}
\left|e(S,T)-\frac{d}{n}|S||T|\right|\le\gl \sqrt{|S||T|} \quad \text{for all $S, T \sub V(G)$} ,
\enq
where $e(S,T)$ is the number of edges with one endpoint in $S$ and the other in $T$, counted twice if both endpoints are in $S \cap T$. From the definition in \eqref{def.expander}, one can easily derive the following fact for any $\gl$-expander $G$ and $A \sub V(G)$ (see \cite[Proposition 2.3]{peled2013lipschitz}):
\beq{expansion} |N(A)| \ge \min\left\{\frac{n}{2}, \frac{d^2}{4\gl^2}|A|\right\},\enq
where $N(A):=\cup_{v \in A} N(v)$. Note that \eqref{expansion} provides a quantified (lower) bound on the expansion of a set $A \sub V(G)$.

One of the main results in \cite{peled2013lipschitz} was the following:
\begin{thm}\label{thm.PSY} Let $G$ be a connected $\gl$-expander where $\gl\le\frac{d}{32(M+1)\log(9Md^2)}$, and $v_0 \in V(G)$. If $f \in_R\Lip_{v_0}(G;M)$, then 
\beq{range.ub'} R(f)=O\left(M\left\lceil\frac{\log\log n}{\log(d/\gl)}\right\rceil\right)\enq
    w.h.p. as $n \ra \infty.$
\end{thm}

We note that \Cref{MT_M-Lips} provides a lower bound on $R(f)$ that matches the upper bound in 
\eqref{range.ub'} up to a constant factor. More precisely, an immediate consequence of \Cref{MT_M-Lips} is that if the given graph $G$ admits the property that for each $A \sub V(G)$, $|N(A)| \lesssim (d/\gl)^2|A|$ (\textit{cf.} \eqref{expansion}), then, w.h.p., $R(f)=\gO\left(\frac{M\log\log n}{\log(d/\gl)} \right)$ (\textit{cf.} \eqref{range.ub'}) for $f \in_R \Lip_{v_0}(G;M)$. 

We also reiterate that \Cref{MT_M-Lips} holds for arbitrary $M$, even for $M \ra \infty$, which we think is interesting in the following context. It is easy to see that $\gl$ cannot be significantly smaller than $\sqrt d$ for any $\gl$-expanders (e.g. plug in any single vertex $w$ for $S$ and $N(w)$ for $T$), by which the validity of \Cref{thm.PSY} is limited to $M= O(\sqrt d/\log d)$. Peled et al asked whether a result similar to \Cref{thm.PSY} continue to hold for larger $M$, or even for $f \in_R \Lip_{v_0}(G;\infty)$. The upper bound on $\lambda$ in the hypothesis of the above theorem is relaxed to $\gl\le d/5$ in a recent work of Krueger, Li, and the second author \cite{krueger2024lipschitz}, but still a certain upper bound assumption on $M$ was required there (e.g. for $\lambda$ close to $\sqrt d$, it was required that $M= \min\{O(d/\log d), (\log n)^{O(1)}\}$). A similar restriction on $M$ also appeared in \cite{butler2024local} that study random $M$-Lipschitz functions on regular trees. All of those results leave it as an interesting question whether the behavior of the range of random Lipschitz functions on expanders would be different for $M \gg d$. We note that, in \cite{peled2013grounded}, where the behavior of "grounded Lipschitz functions on trees" is considered, it was suspected that random Lipschitz functions will exhibit different behaviors when $M \gg d$. We refer interested readers to the last paragraph of \cite[Section 2]{peled2013grounded} for more details.

\Cref{MT_M-Lips} implies an analogous result for $f \in_R\Lip_{v_0}(G;\infty)$ as below.

\begin{thm}\label{MT_Lips}
    There is a constant $c>0$ for which the following holds. Let $G$ be a connected graph on $n$ vertices. Fix $v_0 \in V(G)$, and let $f$ be chosen uniformly at random from $\Lip_{v_0}(G;\infty)$. Then for any $r$ that satisfies $\max\{|B_r(v)|:v \in V(G)\}<c\log n$,
    \[\pr\left[R(f)< \frac{r}{2}\right] =o_n(1).\]
\end{thm}

\nin Again, we can take $c$ slightly smaller than $1/2$.

In fact, \cite[Theorem 2.1]{benjamini2007random} already proves a version of \Cref{MT_M-Lips} for random $\Z$-homomorphisms, and it was noted in \cite{peled2013lipschitz} that "It appears that by applying the techniques in \cite{benjamini2007random} one may obtain a converse to \Cref{thm.PSY}, showing that $\mathbb E[\max(f)] \ge cM\log\log n$ for some $c$ depending only on $d$ and $\gl$." However, as far as we could see, a straightforward application of the method in \cite{benjamini2007random} gives \Cref{MT_M-Lips} only for bounded $M$ (see \Cref{rmk}), and so our main task was to derive the same conclusion for arbitrary $M$ (and thus for $f \in \Lip_{v_0}(G;\infty)$). \Cref{subs.pf sketch} provides a brief proof sketch of \Cref{MT_M-Lips}; \Cref{MT_Lips} is almost an immediate consequence of \Cref{MT_M-Lips}.

As corollaries, we obtain that the range of a random $M$-Lipschitz function from a graph of ``small" degree is ``large."

\begin{cor}\label{Cor_M-Lips}
There is a constant $c>0$ for which the following holds.    Let $G$ be a connected graph on $n$ vertices with maximum degree $d=d(n)$. Fix $v_0 \in V(G)$ and let $f$ be chosen uniformly at random from $\Lip_{v_0}(G;M)$. Then,
    \[\pr\left[R(f)< \frac{cM\log\log n}{\log d} \right] =o_n(1).\]
\end{cor}

\begin{cor}\label{Cor_Lips}
There is a constant $c>0$ for which the following holds. Let $G$ be a connected graph on $n$ vertices with maximum degree $d=d(n)$. Fix $v_0 \in V(G)$ and let $f$ be chosen uniformly at random from $\Lip_{v_0}(G;\infty)$. Then,
    \[\pr\left[R(f)< \frac{c\log\log n}{\log d}\right]=o_n(1).\]
\end{cor}

\begin{proof}[Derivation of Corollaries \ref{Cor_M-Lips} and \ref{Cor_Lips}] First,  let $G$ and $f$ be as in \Cref{Cor_M-Lips}, and $c$ be the constant in \Cref{MT_M-Lips}. Observe that, by the maximum degree condition,
\beq{tree}
\max\{|B_r(v)|: v \in V(G)\} \leq (d+1)^r.\enq
Therefore, we have $\max\{|B_{r-1}(v)|:v \in V(G)\} \leq c\log n$ for some $r=\Omega(\log\log n/\log d)$. Now, the conclusion of \Cref{Cor_M-Lips} follows from Theorem \ref{MT_M-Lips}.

The conclusion of \Cref{Cor_Lips} follows similarly from \Cref{MT_Lips}.    
\end{proof}

\subsection*{Graphs with logarithmic degree.} In this section, we present our next main result, \Cref{MT_torus}.

The conclusions of Corollaries \ref{Cor_M-Lips} and \ref{Cor_Lips} are vacuous if the maximum degree of $G$ dominates $\log n$, which motivates the study of $R(f)$ on graphs with logarithmic degree. 

In \cite{benjamini2007random}, which was the main motivation for \Cref{MT_torus}, Benjamini, Yadin, and Yehudayoff investigated typical range of random $\mathbb Z$-homomorphisms on various graphs with logarithmic degree. Of particular interest was graphs with a large diameter, with the following motivation: Kahn \cite{kahn2001range} showed that the range of random $\mathbb Z$-homomorphisms on $d$-dimensional Hamming cube, $Q_d$, is $O(1)$ w.h.p., confirming a conjecture of Benjamini et al \cite{benjamini2000random} in a very strong way. (Their original conjecture was $o(d)$; the constant $O(1)$ in \cite{kahn2001range} was subsequently improved to 5 by Galvin \cite{galvin2003homomorphisms}.) The Hamming cube has a logarithmic degree (that is, the degree $d=\log n$ where $n=|V(Q_d)|$) but it also has a quite small diameter compared to the number of vertices which could have been the reason for the small range. So a natural question here is whether a graph with logarithmic degree and a large diameter could have a small range.

With this motivation, one specific family of graphs considered in \cite{benjamini2007random} is $C_{n,k}$, where $n, k \in \mathbb N$ and $n$ is even (see \Cref{adjacent}): the vertex set of $C_{n,k}$ is $[n] \times [k]$, and two vertices $(u,u')$ and $(w,w')$ are adjacent iff $|u-w|=1$. In other words, it is a cycle of $n$ layers, where each layer has $k$ vertices and is connected to both its adjacent layers by a complete bipartite graph. In particular, $C_{n,k}$ can have an arbitrarily large diameter by the choice of the value of $n$.
\begin{figure}[h]
    \begin{center}
    \begin{tikzpicture}[scale=0.5]
        % Parameters
        \def\n{4} % Number of vertices on the left side
        \def\m{4} % Number of vertices on the right side
        \def\l{4} % Number of vertices on the right side
        \def\v{1.5} % Vertical distance between the vertices
        \def\h{2.5} % Horizontal distance between the two sets of vertices
        \def\z{5} % Horizontal distance between the two sets of vertices
        % Draw left vertices
        \foreach \i in {1,...,\n} {
            \node (l\i) at (0, -\i*\v) [circle, fill=black, inner sep=1.5pt] {};
        }
        % Draw right vertices
        \foreach \j in {1,...,\m} {
            \node (r\j) at (\h, -\j*\v) [circle, fill=black, inner sep=1.5pt] {};
        }
        % Draw right vertices
        \foreach \k in {1,...,\l} {
            \node (y\k) at (\z, -\k*\v) [circle, fill=black, inner sep=1.5pt] {};
        }
        % Draw edges between left and right vertices
        \foreach \i in {1,...,\n} {
            \foreach \j in {1,...,\m} {
                \draw (l\i) -- (r\j);
            }
        }
        \foreach \j in {1,...,\m} {
            \foreach \k in {1,...,\l} {
                \draw (r\j) -- (y\k);
            }
        }
    \end{tikzpicture}
    \end{center}
    \caption{Three adjacent layers of $C_{n,k}$ for $k=4$}
    \label{adjacent}
\end{figure}
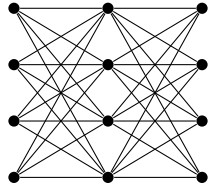

In \cite{benjamini2007random}, it was shown that there is a sharp transition for the range of random $\mathbb Z$-homomorphisms from $C_{n,k}$. We use $\Hom_{v_0}(G)$ for the family of $\mathbb Z$-homomorphisms $f$ from $G$ with $f(v_0)=0$ for some fixed $v_0 \in V(G)$.

\begin{thm}[Theorems 3.1-3.2 in \cite{benjamini2007random}]\label{BYY_MT} Fix $v_0 \in V(C_{n,k})$, and let $f$ be chosen uniformly at random from $\Hom_{v_0}(G)$. Let $\psi:\mathbb N \ra \mathbb R^+$ be such that $\lim_{n \ra \infty} \psi(n)=\infty$. If $k=2\log n+\psi(n)$, then, as $n \ra \infty$,
\beq{range.ub} \pr[R(f)\le 3] = 1 - o_n(1).\enq
On the other hand, if $k=2\log n-\psi(n)$ and $\psi(n)$ is monotone, then, as $n \ra \infty$,
\[\pr\left[R(f)\ge\frac{2^{\psi(n-2)/4}}{\psi(n)}\right] = 1 - o_n(1).\]
\end{thm}

Notice that, regardless of the value of $k$, there is an obvious way to construct (many) $f$'s in $\Hom_{v_0}(G)$ with $R(f)=3$: one can set 
\[\text{$f \equiv c$ for the even (odd, resp.) layers and $f \in \{c \pm 1\}$ for the odd (even, resp.) layers}\]
(the value of $c$ depends on which layer $v_0$ belongs to). So the first conclusion of the above theorem provides a (tight) sufficient condition under which random $\Z$-homomorphisms exhibit ``no fluctuations."

The proof of \Cref{BYY_MT} crucially uses the fact that, under the consideration of $\Z$-homomorphisms, all the layers must be ``monochromatic" (i.e., $f$ takes a single value) or ``bi-chromatic." This special property enables one to use induction on the number of layers. This approach, however, doesn't seem to extend to $M$-Lipschitz functions where each layer has (far) more diverse options. Nonetheless, the next theorem provides a sufficient condition for random $M$-Lipschitz functions on $C_{n,k}$ exhibit no fluctuations, a behavior similar to \eqref{range.ub}. (Note that we can produce many $M$-Lipschitz functions $f$ with $R(f)=M+1$ by setting $f \in [a, a+M]$ for some $a$, where for two integers $b<c$, $[b,c]:=\{b, b+1, \ldots, c\}$.)
\begin{thm}\label{MT_torus}
    There is a constant $C$ for which the following holds. For any $\gamma >C$, if $k> \gamma M^2\log (Mn)$, then for $f$ chosen uniformly at random from $\Lip_{v_0}(C_{n,k};M)$,
    \beq{Lips.range.ub} \pr\left(R(f) \le M+1\right)= 1-O(1/\gamma). \enq
Furthermore, if $\log M=o_n(\log n)$, then
    \[ \pr\left(R(f) \le M+1\right)= 1- o_n(1). \]
\end{thm}

We remark that \eqref{Lips.range.ub} cannot hold true for arbitrarily large $M$; for a simple illustration, for large $M$ (say, $M \gg V(C_{n,k})$), the event that all vertices take values from $[0,M]$ and the event that the layers take values alternately $[0, M+1]$ and $[1,M]$ are almost equally likely. We suspect that the assumption in \Cref{MT_torus} may be improved to $k=\omega(M\log (Mn))$, which would be tight (in the order of magnitude) if true as shown in the example below.

\begin{example}  If $(M \ll)~k \le \frac{M\ln(Mn)}{2}$, then the conclusion of \Cref{MT_torus} doesn't hold. To see this, let $\alpha$ be the number of $M$-Lipschitz functions in which all vertices take values from $[0,M]$, and $\beta$ be that with one vertex takes the fixed value $M+1$ and its neighbors take values from $[1,M]$ (and all the other vertices are in $[0,M]$). Then $\beta/\alpha= ((n-2)k-1)\left(\frac{M}{M+1}\right)^{2k} \gg 1$.
\end{example}

The proof of \Cref{MT_torus} adapts the beautiful entropy argument of Kahn \cite{kahn2001range}, but there were two main obstacles to overcome. First, \cite{kahn2001range} applies union bound of some exceptional probability over a cycle that traverses all the layers of $Q_d$, so indeed the small diameter of $Q_d$ plays a role in the proof. This property is absent in $C_{n,k}$. Second, more substantially, the proof of \cite{kahn2001range} also heavily uses the fact that, for any vertex $v$, $N(v)$ is either monochromatic or bi-chromatic. Kahn's proof was extended to ``multicolor cases" (and more) by Engbers-Galvin \cite{engbers2012h}, but their proof crucially assumes that the number of given colors is constant, while in the setting of \Cref{MT_torus}, $R(f)$ (``the number of colors") can be as large as linear in both $M$ and $n$.

Finally, we note that for the direction that corresponds to the second conclusion of \Cref{BYY_MT}, immediate applications of Theorems \ref{MT_M-Lips} and \ref{MT_Lips} provides a sufficient condition for the range of random Lipschitz functions on $C_{n,k}$ to be large:

\begin{cor} \label{cor1}
    There is a constant $C$ for which the following holds. Let $f$ be chosen uniformly at random from $\Lip_{v_0}(C_{n,k};M)$. If $k \le \log n/\gamma$, then
    \[\pr\left[R(f) < C\gamma M\right] =o_n(1).\]
\end{cor}

\begin{cor} \label{cor2}
    There is a constant $C$ for which the following holds. Let $f$ be chosen uniformly at random from $\Lip_{v_0}(C_{n,k};\infty)$. If $k\le \log n/\gamma$, then
    \[\pr\left[R(f) < C \gamma \right] =o_n(1).\]
\end{cor}

\nin\textit{Organization.} In 
\Cref{sec.prelim}, we collect some basics about the entropy function for its use in \Cref{sec.torus}. The main results, \Cref{MT_M-Lips}, \Cref{MT_Lips} and \Cref{MT_torus} are proved in \Cref{sec.M-Lips},   \Cref{sec.Lips}, and \Cref{sec.torus}, respectively.

\mn\textit{Notation.} For integers $a$ and $b$, $[a]=\{1, 2, \ldots, a\}$ and $[a,b]=\{a, \ldots, b\}$. As usual, $u \sim w$ means $\{u,w\} \in E$, and for $W \sub V$, $u \sim W$ means $u \sim w$ for some $w \in W$. For $W \sub V$, $G[W]$ denotes the induced subgraph of $G$ on $W$. We use $\dist(u,v)$ for the length of a shortest path between $u$ and $v$ in $G$. Our asymptotic symbols are standard, and in particular, we use $O(f)$ if the quantity of interest is bounded by $C\cdot f$ for some absolute constant $C$.

\section{Entropy basics} \label{sec.prelim}
For $p \in [0,1]$, we define
\[H(p)=p\log\left(\frac{1}{p}\right)+(1-p)\log\left(\frac{1}{1-p}\right),\]
where $0\log\frac{1}{0} := 0$.

For a discrete random variable $X$, the \emph{(binary) entropy} of $X$ is
\[ H(X) := \sum_x \pr(X=x) \log\frac{1}{\pr(X=x)}.\]
For discrete random variables $X$ and $Y$, we define
\[ H(X|Y=y) = \sum_x \pr(X=x|Y=y) \log\frac{1}{\pr(X=x|Y=y)} .\]
The \emph{conditional entropy} of $X$ given $Y$, denoted by $H(X|Y)$, is
\beq{conEntro}
H(X|Y) = \sum_y \pr(Y=y) H(X|Y=y) = \sum_y \pr(Y=y) \sum_x \pr(X=x|Y=y) \log\frac{1}{\pr(X=x|Y=y)} .
\enq

Below are some standard facts about entropy:

\begin{prop}\label{prop.ent} The following holds for discrete random variables $X$, $Y$, $Z$, and $X_1, \dots X_n$.
\begin{enumerate}[(a)]
\item $H(X) \le \log |\mathrm{Image}(X)|$, where $\mathrm{Image}(X)$ is the set of values $X$ takes with positive probability, with equality iff $X$ is uniform from its image.
\item $H(X,Y)=H(X)+H(Y|X)$
\item $H(X|Y) \le H(X)$
\item if $Z$ is determined by $Y$, then
$H(X|Y) \le H(X|Z)$ and $H(Y,Z|X) = H(Y|X)$.
\end{enumerate}
\end{prop}

We also need the following version of \emph{Shearer's Lemma} \cite{chung1986some}.

\begin{lem}\label{lem.Sh}
Let $X=(X_1, \dots, X_N)$ be a random vector, and suppose $\alpha:2^{[N]}\rightarrow \mathbb R^+$
satisfies
\[ \sum_{A\ni i}\alpha_A =1 \quad \forall i\in [N].\]
Then for any partial order $\prec$ on $[N]$, we have
\[ H(X)\leq \sum_{A\subseteq [N]}\alpha_AH\big(X_A|(X_i:i \prec A)\big),\]
where $X_A=(X_i:i\in A)$ and $i\prec A$ means $i\prec a$ for all $a\in A$.
\end{lem}

\section{Proof of \Cref{MT_M-Lips}}\label{sec.M-Lips}

We first introduce some notations that will be used in this section. Denote by $B_r(v) ~(\sub V(G))$ the \textit{ball of radius $r$ centered at $v$}, that is, $B_r(v):=\{u \in V(G):\dist(u,v)\le r\}$. We say the ball $B_r(v)$ has \textit{exact radius} $r$ if their exists a vertex $u \in B_r(v)$ such that dist$(u, v) = r$. Define the $i^{th}$ layer of $B_r(v)$ as $L_i(v) = B_{r-i}(v) \setminus B_{r-i-1}(v)$. When $B_r(v)$ has exact radius $r$ (and $r$ and $v$ are understood), we will write $\Gamma := L_0(v)$ for the \textit{boundary} of $B_r(v)$ and  $B^o$ for $B_{r-1}(v)=B_r(v) \setminus \Gamma$. %See \Cref{figure1} for a demonstration.

For any function $f$ and a subset $W$ of the domain of $f$, $f|_W$ is used for $f$ with its domain restricted on $W$.

\subsection{Proof sketch}\label{subs.pf sketch} We follow the proof strategy of Benjamini, Yadin, and Yehudayoff \cite{benjamini2007random} who showed a similar result for random $\Z$-homomorphisms. Roughly, they showed that 

(i) there are "many" disjoint balls of exact radius $r$ in $V(G)$ (\Cref{many balls});

(ii) for each of the balls (say its center is $w$) and $f \in_R \Hom_{v_0}(G)$, the probability that $f(w)$ is "high" is bounded from below by some quantity. 

\nin Combining (i) and (ii), they concluded that there is a good chance that some center of the balls have a high value of $f$.

We note that this approach extends to random $M$-Lipschitz functions almost immediately if $M$ is not too large, since the lower bound in (ii) above still carries over to $M$-Lipschitz functions for such $M$. However, the lower bound decays as $M$ grows, so, in particular, extending this result to $M \ra \infty$ requires more careful arguments. These are the contents of \Cref{sec.M-LIps_pf}.

We close this section by recalling a lemma from \cite{benjamini2007random} which constitutes the first step of the proof.

\begin{lem}[Claim 2.5 in \cite{benjamini2007random}] \label{many balls}
    For any $r \in \mathbb N$ the following holds. Let $m=m(G,r)$ be the maximum size of a ball of radius $r$ in $G$. For any $W \sub V(G)$, there exists $U \subseteq W$ of size $|U| \geq \left\lfloor \frac{|W|}{m^2} \right\rfloor$ such that 
    \begin{enumerate}
        \item for all \( u \in U \), the ball \( B_r(u) \) (in $G$) is of exact radius $r$; and
        \item for all distinct \( u, u' \in U \), we have \( B_r(u) \cap B_r(u') = \emptyset \).
    \end{enumerate}
\end{lem}

\subsection {Proof of \Cref{MT_M-Lips}}\label{sec.M-LIps_pf} 

Our key result in this section is the lemma below which provides a lower bound on the probability that the center of a ball admits a high value of $f$. 

    \begin{lem} \label{lower prob}
There exists a constant $C$ for which the following holds. Let positive integers $M$ and $r$ be given. Let $B = B_{r}(v)$ be a ball of exact radius $r$ and $v_0 \notin B$. Let $g_0 \in \Lip_{v_0}(G[V \setminus B_{r-1}(v)]; M)$ be an $M$-Lipschitz function that can be extended to an $M$-Lipschitz function in $\Lip_{v_0}(G;M).$ Define
        \[\cG = \{g \in \Lip_{v_0}(G[B]; M): g|_{\Gamma} = g_0|_{\Gamma}\};\]
        and
        \[\cG' = \{g \in \Lip_{v_0}(G[B]; M):\max(g)-\min(g|_\Gamma) \geq \frac{Mr}{2} \text{ and } g|_{\Gamma} = g_0|_{\Gamma}\}.\]
        Then, for $m = m(G, r)$ the maximum size of a ball of radius $r-1$ in $G$, 
        \beq{height.lb} \frac{|\cG'|}{|\cG|} \geq C^{-m}.\enq
    \end{lem}

\begin{remark}
    If $g_0|_\Gamma$ takes a single value, then the conclusion of \Cref{lower prob} is almost immediate: for simplicity of exposition, assume that $g_0|_\Gamma \equiv 0$ and $M$ is even. Notice that $\cG'$ contains $f$ with $f(v) \in \left[\frac{(i-1)M}{2},\frac{iM}{2}\right]$ for all $v \in L_i$ $(i \ge 1)$, so $|\cG'|\ge (M/2)^{|B_{r-1}(v)|}$; on the other hand, $|\cG|\le (2M+1)^{|B_{r-1}(v)|}$ because
    \beq{because} \text{$f$ is $M$-Lipschitz, so knowing $f(L_i)$ leaves at most $2M+1$ choices for each vertex in $L_{i-1}$.}\enq So
for this special case, we easily have
\[\frac{|\cG'|}{|\cG|} \ge \left(\frac{M}{2(2M+1)}\right)^{|B_{r-1}(v)|} \ge O(1)^{-m}.\]
Our main contribution for the proof of \Cref{MT_M-Lips} is to prove the same lower bound under an arbitrary boundary condition $g_0|_\Gamma$.
\end{remark}

\begin{proof}[Proof of \Cref{MT_M-Lips} assuming \Cref{lower prob}]

Let $r$ be as in \Cref{MT_M-Lips} (so, in particular, $m$ in the statement of \Cref{lower prob} is at most $c\log n$, where we get to choose $c$), and set $k=\lfloor n/m^2\rfloor -1$. By \Cref{many balls}, there is a collection $\{B_1, \ldots, B_k\}$ of pairwise disjoint balls of exact radius $r$ in $G$ such that $v_0 \notin \cup_{i \le k} B_i$. For notational simplicity, let $\hat B=\cup_{i \le k} B_i^o$.

 Define 
    \[
        H = \{h \in \Lip_{v_0}(G[V\setminus \hat B]; M) : \exists f \in \text{Lip}_{v_0}(G ; M) \text{ such that } f|_{G[V \setminus \hat B]} = h\}
    \]
(that is, $H$ is the family of $M$-Lipschitz functions on $G[V \setminus \hat B]$ that can be extended to an $M$-Lipschitz function on $G$). 

Then, for $f \in_R \Lip_{v_0}(G;M)$,
    \beq{ub.precise}\begin{split}
        \pr\left[R(f) < \frac{Mr}{2}\right] &\leq \pr \left[\bigwedge_{i \in [k]} \left\{R(f|_{B_i}) < \frac{Mr}{2}\right\} \right] \\   
        &= \sum_{h \in H} \pr\left[f|_{G[V \setminus \hat B]}=h\right] \pr\left[\bigwedge_{i \in [k]} \left\{R(f|_{B_i}) < \frac{Mr}{2}\right\} \big\rvert \left\{ f|_{G[V \setminus \hat B]} = h\right\}\right]\\
        &\stackrel{(\dagger)}{=} \sum_{h \in H} \pr\left[f|_{G[V \setminus \hat B]}=h\right] \prod_{i=1}^{k} \pr\left[R(f|_{B_i}) < \frac{Mr}{2} \big\rvert \left\{ f|_{G[V \setminus \hat B]} = h\right\}\right]\\
        &\stackrel{\eqref{height.lb}}{\leq} \left(1 - C^{-m}\right)^{k} \le \exp\left(-kC^{-m}\right)=o_n(1). \end{split}\enq
where $(\dagger)$ uses the fact that all $B_i$'s are disjoint so the events $\left\{R(f|_{B_i})<\frac{Mr}{2}\right\}$ are mutually independent; the last equality follows by choosing $c$ (in the statement of \Cref{MT_M-Lips}) small enough.\end{proof}

The rest of this section will be devoted to proving \Cref{lower prob}. We first focus on the case $M \ge 2$, and the (easy) case $M=1$ will be proved at the end of this section. In order to work with an arbitrary boundary condition $g_0|_\Gamma$, we further partition $\cG$ and $\cG'$ using the following definitions. Notice that in \Cref{lower prob}, without loss of generality, we may assume that
\beq{zero}\min(g_0|_\Gamma)=0.\enq
Let a ball $B=B_r(v)$ of exact radius $r$ be given, and consider an $M$-Lipschitz function $\hat g$ on $B$. For $i \ge 0$, let $$A_{\hat g, i} = \left\{u \in L_i: \hat g(u) \geq \frac{M(i+1)}{2}\right\},$$
and set $A_{\hat g}=\cup_{i=0}^r A_{\hat g,i}$. 
Let $Q_{\hat g} = \{u \in B \setminus A_{\hat g}: u \sim A_{\hat g}\}$,
\[Q^*_{\hat g,i}= \left\{u \in Q_{\hat g} \cap L_i: \exists w \in A_{\hat g} \text{ such that } u \sim w \text{ and } \hat g(w)>\frac{M(i+2)}{2}\right\},\]
and $Q^*_{\hat g}=\cup_{i=0}^r Q^*_{\hat g, i}$. Now we partition the collection of $M$-Lipschitz functions on $B$ using the notion introduced above: given a ball $B$ of exact radius $r$ and $A, Q^* \sub B$, let 
\[\cG(B, A, Q^*)=\{f\in \Lip_{v_0}(G[B];M):A_f=A, Q^*_f=Q^*\}.\]
Say $f, f' \in \cG(B,A,Q^*)$ are \textit{equivalent} if $f \equiv f'$ on $A \cup Q^* \cup \Gamma$. Below lemma shows that for each equivalent class, the lower bound in \Cref{lower prob} holds. Recall our assumption \eqref{zero}.

    \begin{lem} \label{new lower bound} Let $M \ge 2$. 
Let $B, A$, and $Q^*$ as above be given.    For any $f \in \cG(B,A,Q^*)$, denote by $\cG_f(B,A,Q^*)$ the collection of functions in $\cG(B,A,Q^*)$ equivalent to $f$. Let
        \[\cG'_f(B,A,Q^*)=\left\{f' \in \cG_f(B,A,Q^*): \max(f'|_B)\ge \frac{Mr}{2}\right\}.\]
        Then, with $k = |A \cup Q^* \cup \Gamma|$,
        \[\frac{|\cG'_f(B,A,Q^*)|}{|\cG_f(B,A,Q^*)|} \ge \left(\frac{M-1}{2(2M+1)}\right)^{|B|-k}.\]
    \end{lem} 

    \begin{proof} Fix an $f \in \cG(B, A, Q^*)$. We have the easy upper bound $|\cG_f(B,A,Q^*)| \le (2M+1)^{|B|-k}$ (see \eqref{because}), so it suffices to show that
    \beq{G'.lb}|\cG'_f(B,A,Q^*)| \ge \left(\frac{M-1}{2}\right)^{|B|-k}.\enq
Our plan is to construct many $M$-Lipschitz functions on $B$ that belong to $\cG'_f(B, A, Q^*)$. To this end, consider $g$ on $B$ that satisfies the following properties:
\[g(u) \begin{cases}
    = f(u) & \text{if } u \in A \cup Q^* \cup \Gamma \\
     \in [\frac{Mi}{2}, \frac{M(i+1)}{2}-1] & \text{if } u \in L_i \setminus (A \cup Q^* \cup \Gamma)
\end{cases}\]

\nin Note that the above construction of $g$ produces $\left(\left\lfloor \frac{M(i+1)}{2}-1\right\rfloor -\left\lceil \frac{Mi}{2}\right\rceil +1 \right)^{|B|-k} \ge \left(\frac{M-1}{2}\right)^{|B|-k}$ distinct functions. Therefore, \eqref{G'.lb} will follow from the claim below:

\begin{claim} Any $g$ defined as above belongs to    $\cG'_f(B,A,Q^*)$.
\end{claim}

\begin{subproof}
First of all, we show that $\max(g|_B) \ge \frac{Mr}{2}$. If the center of the ball $v$ does not belong to $A \cup Q^* \cup \Gamma$, then by the definition of $g$, $g(v) \ge \frac{Mr}{2}$; if $v \in A \cup Q^* \cup \Gamma$, then $v \in A \cup Q^*$, so by the definitions of $A$ and $Q^*$ (and the fact that $f$ is $M$-Lipshitz) we have $g(v)=f(v) \ge \frac{Mr}{2}$.

So we are left with showing that $g$ is $M$-Lipschitz on $B$. We show that $|g(u)-g(w)|\le M$ for $u \sim w$ only for $u \in L_i \setminus (A\cup Q^* \cup \Gamma)$ for $i \in [1, r-1]$ (and the affiliation of $w$ varies), as other cases are immediate or similar. Note that, for such $u$, by the definition of $g$,
\beq{u.bound} g(u) \in \left[\frac{Mi}{2},\frac{M(i+1)}{2}-1\right]. \enq

First, suppose $w \notin A \cup Q^* \cup \Gamma$. Since $w \sim u \in L_i$, we have $w \in \cup_{j=i-1}^{i+1}L_j$, so by the definition of $g$, $g(w) \in \left[\frac{M(i-1)}{2}, \frac{M(i+2)}{2}-1\right]$. Combining this with \eqref{u.bound}, we have $|g(w)-g(u)| \le M$.

Next, if $w \in A \cup Q^* \cup \Gamma$, then there are three possibilities; in all of the case analyses below, we will find upper and lower bounds on $f(w)~(=g(w))$, and then compare them to the bound in \eqref{u.bound}.

\nin     \textbf{Case 1: } $w \in \Gamma \cap A^c \cap (Q^*)^c$. In this case, $u \in L_1 \setminus (A \cup Q^* \cup \Gamma)$ (since $u \sim w$), and so $g(u) \in \left[\frac{M}{2}, M-1\right]$. Since $w \notin A$, $0 \stackrel{\eqref{zero}}{\leq} f(w) < \frac{M}{2} $. Thus, $|g(u)-g(w)|=|g(u)-f(w)|\le M$.

\nin     \textbf{Case 2: } $w \in A$.  We first claim that $f(w) \ge \frac{Mi}{2}$; this easily follows from the facts that $w \in \cup_{j=i-1}^{i+1}L_j$ (again, since $u \sim w$) and the definition of $A$. Next, we claim that $f(w) \le \frac{M(i+2)}{2}$: to see this, since $u \notin Q^*$ (while $u \in L_i$), by the definition of $Q^*$, $u$ cannot be adjacent to any vertex $\tilde w \in A$ with $f(\tilde w) >\frac{M(i+2)}{2}$.

By the two claims above, we conclude that $|g(u)-g(w)|=|g(u)-f(w)|\le M$.

\nin     \textbf{Case 3: } $w \in Q^*$.  We first claim that $f(w) > \frac{M(i-1)}{2}$; indeed, since $w \sim u \in L_i$, we have $w \in \cup_{j=i-1}^{i+1} L_i$. Then by the fact that $w \in Q^*$, there is some vertex $\tilde w \in A$ such that $f(\tilde w)>\frac{M((i-1)+2)}{2}=\frac{M(i+1)}{2}$. Now the claim follows since $f$ is $M$-Lipschitz. Next, an upper bound $f(w) < \frac{M(i+2)}{2}$ easily follows from the fact that $w \notin A$ (and, again, $w \in \cup_{j=i-1}^{i+1} L_i$). 

The combination of the two bounds above and \eqref{u.bound} yields that $|g(u)-g(w)|=|g(u)-f(w)|\le M$.
\end{subproof}

This concludes the proof of \Cref{new lower bound}.    
    \end{proof}

\begin{proof}[Proof of \Cref{lower prob} for $M \ge 2$]
    Let $f \in_R \cG$. Then the left-side of \eqref{height.lb} is equal to (suppressing the dependence on $B$)
    \[\pr[f \in \cG'] = \sum_{A, Q^*}{\sum}^* \pr[f \in \tilde \cG'(A, Q^*)|f \in \tilde \cG(A, Q^*)]\cdot \pr[f \in \tilde \cG(A, Q^*)],\]
    where ${\sum}^*$ ranges over all equivalent classes in $\cG(A,Q^*)$, and $\tilde \cG(A,Q^*)$ denotes an individual equivalent class. (The equality follows from the observation that the collection of equivalent classes $\tilde \cG(A,Q^*)$'s for all possible $A, Q^*$ partitions $\cG$.) By \Cref{new lower bound}, we have
    \[\pr[f \in \tilde \cG'(A, Q^*)|f \in \tilde \cG(A, Q^*)]\ge \left(\frac{M-1}{2(2M+1)}\right)^{|B|-|A \cup Q^* \cup \Gamma|} \ge \left(\frac{M-1}{2(2M+1)}\right)^{m}\]
    for any $\tilde \cG(A,Q^*)$, which yields the conclusion.    
\end{proof}

The construction in the proof of \Cref{new lower bound} requires $M\ge 2$. For $M=1$, we use the following construction; this is essentially the same as the proof of \cite[Lemma 2.4]{benjamini2007random}, but we present it here for ease of reference.

\begin{proof}[Proof of \Cref{lower prob} for $M=1$] Observe that, in this case,
\[|\cG|\le (2M+1)^m=3^m,\]
so it suffices to simply show that $\cG'\ne \emptyset$, that is, there exists a 1-Lipschitz function $g$ such that $g|_\Gamma=g_0|_\Gamma$ and $\max(g) \ge r/2$ (assuming \eqref{zero}). Let $f$ be any 1-Lipschitz function on $G[B]$ such that $f|_\Gamma=g_0|_\Gamma$ (the existence of such an $f$ was assumed in \Cref{lower prob}). Furthermore, we can assume that $f$ is non-negative on $G[B]$, since taking $|f|$ doesn't break the Lipschitz property of $f$ (and the boundary condition remains the same since $f|_\Gamma \ge 0$ by \eqref{zero}).  

We construct a sequence of 1-Lipschitz functions $\{h_i\}_{i \ge 0}$ via the following recursive process starting from $h_0:=f$. At the $i$th step ($i \ge 1$), define $h_i$ as follows: for all $u \in B$,
\[h_i(u) = \begin{cases}
    h_{i-1}(u) & \text{if } u \notin B_{r-i}; \\
     h_{i-1}(u) & \text{if } u \in B_{r-i} \text{ and } h_{i-1}(u) \ne i-1; \\
     i & \text{if } u \in B_{r-i} \text{ and }
 h_{i-1}(u)=i-1 \end{cases}\]
We claim that $g(u):=h_r(u)$ is 1-Lipschitz which will finish the proof (since $h_r|_\Gamma=g_0|_\Gamma$, and $h_r$ takes the value $r$ at the center of $B$). It suffices to show that, assuming that $h_{i-1}$ is 1-Lipschitz, $h_i$ is 1-Lipschitz. 

To this end, crucially observe that (by the construction), for each $i$,
\beq{crucial}h_i(u) \ge j \text{ if } u \in B_{r-j} \text{ and } j\le i.\enq
Now, suppose $u \sim w$. The only case needed to check is when $h_i(u) \ne h_{i-1}(u)$ and $h_i(w)=h_{i-1}(w)$. Suppose $u \in L_\ell$ for some $\ell$. Since $h_i(u) \ne h_{i-1}(u)$, we have $\ell \ge i$ (otherwise $h_i(u)$ wouldn't have changed). Then $w \in \cup_{j=\ell-1}^{\ell+1} L_j$, so in particular, $h_{i-1}(w) \ge i-1$ by \eqref{crucial}. Also, since $u \sim w$ and $h_{i-1}$ is 1-Lipschitz, $h_{i-1}(w)\le h_{i-1}(u)+1=i$. Therefore, $|h_i(u)-h_i(w)|=|i-h_{i-1}(w)|\le 1$.    
\end{proof}

\begin{remark}\label{rmk}
In fact, the above proof gives (a version of) \Cref{lower prob} for any constant $M$, but it doesn't seem to produce $(\gO(M))^m$ many $M$-Lipschitz functions (for large $M$), which is what we would need to extend the above proof to unrestricted $M$.
\end{remark}

\section{Proof of \Cref{MT_Lips}}\label{sec.Lips}

In this section, we derive \Cref{MT_Lips} from \Cref{MT_M-Lips}. The proof is based on the rather obvious fact that the random $M$-Lipschitz model, upon an appropriate scaling, converges to the random $\mathbb R$-Lipschitz model as $M \ra \infty$. More precisely, let $f$ be a uniformly chosen random element of $\Lip_{v_0}(G; \infty)$ and, for a positive integer $M$, $f_M$ be a uniformly chosen random element of $\Lip_{v_0}(G;M)$. A brief proof of the proposition below was given by Peled, Samotij, and Yehudayoff \cite[page 8]{peled2013grounded}.

\begin{prop} \label{convergence}
    $f_M/M$ converges to $f$ in distribution as $M \ra \infty$.
\end{prop}

\begin{proof}[Proof of \Cref{MT_Lips}]
Let $f$ and $f_M$ as above. Since the range $R$ is a continuous function on $\Lip_{v_0}(G;\infty)$, by the Continuous Mapping Theorem and \Cref{convergence},
$$\pr\left(R(f) < \frac{r}{2}\right) = \lim_{M \to \infty} \pr \left( R\left(\frac{f_M}{M}\right) < \frac{r}{2} \right).$$
The right-side of the above is equal to
\[\begin{split}
      \lim_{M \to \infty} \pr \left( \frac{R(f_M)}{M} < \frac{r}{2} \right)
     &= \lim_{M \to \infty} \pr \left( R(f_M) < \frac{Mr}{2} \right) \\
&=o_n(1)
     \end{split}\]
     by \Cref{MT_M-Lips}.
\end{proof}

\section{Proof of \Cref{MT_torus}}\label{sec.torus}

\subsection*{Set-up} In this section, we use $\bff$ for a uniformly random element of $\Lip_{v_0}(C_{n,k};M)$ (to distinguish this from deterministic $f$). Following \cite{benjamini2000random}, we assume $n$ is even.\footnote{The proof for even $n$ in this section may be extended to odd $n$ with a bit of extra argument, but the proof for even $n$ already exhibits most essential ideas, so we omit the proof for odd $n$ to avoid excessive repetition and keep the clarity of the exposition.}

Write $V$ for $V(C_{n,k})$ and $N=|V|=nk$.
Let $L_0$ be the layer that $v_0$ belongs to. For $i \ge 1$, let $L_i$, the $i$-th layer, be the set of vertices $u$ such that $\dist(u,v_0)=i$. Define the order $\prec$ on the set of layers as the following:
\beq{Sh.order} L_{1} \prec L_{0} \prec L_{3} \prec L_{2} \prec L_{5} \prec L_{4} \prec \cdots \enq
That is,  $L_{2i+1} \prec L_{2i} \prec L_{2i+3}$ for $i \ge 0$. For $v \in L_i$ ($i \ge 1$), write $N^+(v)=\{u \sim v:u \in L_{i+1}\}$ and $N^-(v)=\{u \sim v:u \in L_{i-1}\}$. To extend this definition to $i=0$, pick an arbitrary $w \in L_2$ and set $N^+(v)=N^-(w)$ and $N^-(v)=L_1 \setminus N^+(v)$ for all $v \in L_0$. We will write $|v|=i$ if $v \in L_i$. For $x \in V$ and $X \sub V$, we use $f_x$ for the value of $f$ at $x$; $f_X$ for the vector $(f_x)_{x \in X}$; and $\bar f_X := [\min\{f_u:u \in X\}, \max\{f_u:u \in X\}]$.
Below is our key lemma.
\begin{lem} \label{lem.ideal} There is a constant $C$ for which the following holds. For any $\gamma>C$, if $k>\gamma M^2\log(Mn)$, then for $\bff$ chosen uniformly at random from $\Lip_{v_0}(C_{n,k}; M)$, for any $v \in V$,
\beq{nonideal}
\pr(|\bar\bff_{N_v^+}| \ne M+1) = O((\gamma n)^{-1}).
\enq
Furthermore, if $\log M=o_n(\log n)$, then
\beq{nonideal'}
\pr(|\bar\bff_{N_v^+}| \ne M+1) = o_n(n^{-1}).
\enq
\end{lem}
\nin Note that by symmetry of $C_{n,k}$, the above probability is equal for any vertices, and also for the event $\{|\bar\bff_{N_v^-}|\ne M+1\}.$

\begin{proof}[Derivation of \Cref{MT_torus} from \Cref{lem.ideal}]
Say an edge $\{u,v\} \in E(C_{n,k})$ is \textit{ideal} with respect to $f$ if $\bar f_{N_u}=\bar f_{N_v}=\{k, k+1, \ldots, k+M\}$ for some integer $k$. 
Observe that, for any edge $\{u,v\}$, $\{u,v\}$ must be ideal with respect to $f$ if $|\bar f_{N_v^+}|=|\bar f_{N_v^-}|=|\bar f_{N_u^+}|=|\bar f_{N_u^-}|=M+1$. Therefore,
\beq{prob.ideal} \pr(\text{$\{u,v\}$ is ideal with respect to $\bff$}) \ge 1- 4\pr(|\bar\bff_{N_v^+}|\ne M+1).\enq
 Pick $\cC=(v_0, v_1, \ldots, v_{n-1}, v_0)$ a cycle of length $n$, with $v_1 \in N^+(v_0)$ and $v_{n-1} \in N^-(v_0)$, that traverses all the layers. Say $\cC$ is \textit{ideal} with respect to $f$ if all of its edges are ideal with respect to $f$, noting that if $\cC$ is ideal then there is an integer $k$ such that
\[\text{$f(v) \in \{k, k+1, \ldots, k+M\}$ for all $v \in V(C_{n,k})$.}\]
Therefore, for $\bff \in_R \Lip_{v_0}(C_{n,k};M)$,
\[\pr(R(\bff)\le M+1) \ge \pr(\text{$\cC$ is ideal with respect to $\bff$}) \stackrel{\eqref{prob.ideal}}{\ge} 1-4n\cdot \pr(|\bar\bff_{N_v^+}| \ne M+1) \stackrel{\eqref{nonideal}}{=} 1-O(1/\gamma).\]
Furthermore, if $\log n \gg \log M$, then the above is $1-o_n(1)$.

This completes the proof.
\end{proof}

\begin{proof}[Proof of \Cref{lem.ideal}] Recall that $n$ is even, so $C_{n,k}$ is bipartite. Set
\beq{def.eps} \eps=\pr(|\bar\bff_{N_v^+}|\ne M+1) ~ \left(=\pr(|\bar\bff_{N_v^-}|\ne M+1)\right).\enq
 We call $v$ is \textit{even} (\textit{odd}, resp.) if $|v|$ is even (odd, resp.). We will use $\bar f_{N_v}=(c,c')$ as an abbreviation for $(\bar f_{N^+_v}, \bar f_{N^-_v})=(c,c')$. 

Observe that $|\Lip_{v_0}(C_{n,k};M)|\ge (M+1)^{N-1},$ so by \Cref{prop.ent}(a) we have a trivial lower bound
\beq{ent.lb} (N-1)\log(M+1)\le H(\bff).\enq
Now we give an upper bound on $H(\bff)$ using \Cref{lem.Sh}: since one copy of $N^+_v$, one copy of $N^-_v$, and $(2k)$-copies of $\{v\}$ for all even $v \in V$ form a $(2k)$-fold cover of $V$,
\beq{ent1} H(\bff) \le \sum_{\text{$v$ even, $v \ne v_0$}} H(\bff_v|\bff_x:x \prec v)+\frac{1}{2k}\sum_{\text{$v$ even}} H(\bff_{N_v}|\bff_x:x \prec N_v),\enq
where $x \prec y$ means the layer that $x$ belongs to comes before the layer that $y$ belongs to in the order in \eqref{Sh.order}; $x \prec A$ means $x \prec y$ for all $y \in A$.

Each summand in the first term on the right-hand side of \eqref{ent1} is at most (using \Cref{prop.ent}(d) and the definition of conditional entropy)
\[\begin{split} H(\bff_v|\bff_{N_v})\le  H(\bff_v|\bar\bff_{N_v}) & = \sum_{c, c'} H(\bff_v|\bar\bff_{N_v}=(c,c'))\cdot\pr(\bar \bff_{N_v}=(c, c'))\\
&\le \sum_{c,c'} \log(2M+2-|c\cup c'|)\cdot \pr(\bar\bff_{N_v}=(c,c')),\end{split}\]
where the last inequality uses the fact that knowing $\bar f_{N_v}=(c,c')$ leaves at most $2M+2-|c\cup c'|$ choices for $f_v$ (and \Cref{prop.ent}(a)).

To bound the second term on the right-side of \eqref{ent1}, observe that, using \Cref{prop.ent} and the definition of conditional entropy,
\beq{ent2.1}\begin{split}H(\bff_{N_v}|\bff_x:x \prec N_v) &=H(\bff_{N_v}, \bar \bff_{N_v}|\bff_x:x \prec N_v)\\
& \le H(\bar\bff_{N_v}|\bff_x:x \prec N_v)+H(\bff_{N_v}|\bar\bff_{N_v}, (\bff_x:x \prec N_v))\\
&\le H(\bar\bff_{N_v}|\bff_x:x \prec N_v)+ H(\bff_{N_v}|\bar\bff_{N_v})\\
&= H(\bar\bff_{N_v}|\bff_x:x \prec N_v)+\sum_{c,c'} k(\log|c|+\log|c'|)\cdot \pr(\bar\bff_{N_v}=(c,c')).\end{split}\enq
We further upper bound the first term of the right-side of \eqref{ent2.1} for $|v|\le 2$ by
\[H(\bar\bff_{N_v}|\bff_x:x \prec N_v)=O(\log M),\]
because the number of choices for each of $\max\bar\bff_{N_v}$ and $\min\bar\bff_{N_v}$ is at most $O(M)$ for such $v$.

In sum, the right-side of \eqref{ent1} is bounded above by
\beq{ent2} \begin{split} O\left({\log M}\right)&+ \frac{1}{2k}\sum_{|v|\ge 4, \text{ $v$ even}} H(\bar\bff_{N_v}|\bff_x:x\prec N_v)\\
&+\sum_{\text{$v$ even}} \sum_{c,c'} \log\left[(2M-|c\cup c'|+2)(|c||c'|)^{1/2}\right]\cdot \pr(\bar\bff_{N_v}=(c,c')).\end{split}\enq
We provide an upper bound on the last term of \eqref{ent2} as follows: for any pair $(c,c')$, we have
\[\log\left[(2M-|c\cup c'|+2)(|c||c'|)^{1/2}\right]\le \log(M+1)^2\]
(the maximum is achieved when $c=c'=[k, M+k]$ for some $k$);
on the other hand, if $|c| \ne M+1$ or $|c'| \ne M+1$, then
\[\log\left[(2M-|c\cup c'|+2)(|c||c'|)^{1/2}\right]\le\log M(M+2)\]
(the worst case is either $c=c'=[k,(M-1)+k]$ or $c=c'=[k,(M+1)+k]$ for some $k$).
Therefore, recalling the definition in \eqref{def.eps},
\[\sum_{c,c'} \log\left[(2M-|c\cup c'|+2)(|c||c'|)^{1/2}\right]\cdot \pr(\bar\bff_{N_v}=(c,c'))\le (1-\eps)\log(M+1)^2+\eps\log M(M+2).\]
By plugging the above into \eqref{ent2} and combining with the lower bound in \eqref{ent.lb}, we have
\beq{ent3}\begin{split} (N-1)\log(M+1) \le &~ O(\log M) +\frac{1}{2k}\sum_{|v|\ge 4, \text{$v$ even}} (\bar\bff_{N_v}|\bff_x:x\prec N_v)\\
& +\frac{N}{2}\left[(1-\eps)\log(M+1)^2+\eps\log M(M+2)\right].\end{split}\enq

We first loosely show that $\eps$ is quite small, and then give a tighter upper bound on $\eps$ using the fact that $\eps$ is small. To this end, observe that for $v$ with $|v|\ge 4$,
\beq{nbd.ent} H(\bar\bff_{N_v}|\bff_x:x \prec N_v) \le H(\bar\bff_{N_v}| \bar\bff_{L_{|v|-4}}) =\sum_c H(\bar\bff_{N_v}|\bar\bff_{L_{|v|-4}}=c)\pr(\bar\bff_{L_{|v|-4}}=c) =O(\log M)\enq
(the final equality is again because the number of choices for each of $\max\bar\bff_{N_v}$ and $\min\bar\bff_{N_v}$ is at most $O(M)$ given $\{\bar\bff_{L_{|v|-3}}=c\}$). Combining this with \eqref{ent3}, we have
\[(N-1)\log(M+1)\le O(\log M) + O(n\log M)+\frac{N}{2}\left[(1-\eps)\log(M+1)^2+\eps\log M(M+2)\right],\]
and by simplifying the above, we have
\[\eps = O\left(\frac{M^2\log M}{k}\right),\]
which is small by the assumption that $k>\gamma M^2\log(Mn)$ (and $\gamma >C$ where can choose a large enough $C$).

In order to tighten our analysis, we introduce some more events. First, for $u \in V$, define
\[\text{$Q_u$ to be the event that $\{\text{$|\bar \bff_{N^+_u}|=M+1$ and $|\bar \bff_{N^-_u}|=M+1$}\}$,}\]
noting that 
\beq{neg.Qu.small} \pr(\neg Q_u)\le 2\eps \text{ for any $u$}.\enq
Next, for $u, u' \in V$ with $\dist(u, u')=4$, let 
\[\text{$Q_{u, u'}$ be the event that $\{\bar\bff_{N_u}=\bar\bff_{N_{u'}}\}$.}\]
Finally, for each even $v$ with $|v| \ge 4$, pair it with a vertex $w=w(v)$ with $|w|=|v|-4$ and $\dist(v,w)=4$. Observe that, using \Cref{prop.ent},
\beq{ent4} \begin{split}H(\bar\bff_{N_v}|\bff_x:x \prec N_v) \le H(\bar\bff_{N_v}|\bff_{N_w}) &\le H(\bar\bff_{N_v}, \mathbf{1}_{Q_{v,w}}|\bff_{N_w})\\
& \le H(\mathbf{1}_{Q_{v,w}}|\bff_{N_w})+H(\bar\bff_{N_v}|\mathbf{1}_{Q_{v,w}}, \bff_{N_w})\\
&\le H(\mathbf{1}_{Q_{v, w}}|\mathbf{1}_{Q_w})+H(\bar\bff_{N_v}|\mathbf{1}_{Q_w}, \mathbf{1}_{Q_{v,w}}, \bff_{N_w}).\end{split}\enq
We bound each term on the right-side of \eqref{ent4} in the two claims below.

\begin{claim}\label{claim1}
 $H(\mathbf{1}_{Q_{v,w}}|\mathbf{1}_{Q_w})\le 2\eps+H(32\eps)$.
\end{claim}

\begin{subproof} Note that
\[H(\mathbf{1}_{Q_{v,w}}|\mathbf{1}_{Q_w})=H(\mathbf{1}_{Q_{v,w}}|Q_w)\pr(Q_w)+H(\mathbf{1}_{Q_{v,w}}|\neg Q_w)\pr(\neg Q_w).\]
The second term of the above is 
\[H(\mathbf{1}_{Q_{v,w}}|\neg Q_w)\pr(\neg Q_w) \le \pr(\neg Q_w)\stackrel{\eqref{neg.Qu.small}}{\le} 2\eps.\]
The first term is
\beq{neg.vw} H(\mathbf{1}_{Q_{v,w}}|Q_w)\pr(Q_w)\le H(\mathbf{1}_{Q_{v,w}}|Q_w) =H(\pr(\neg Q_{v,w}|Q_w)), \enq
and we claim that $\pr(\neg Q_{v,w}|Q_w) \le 32\eps$, which will conclude the proof (combined with the fact that $\eps$ is small, so $32\eps \in [0,1]$). To this end, first notice that $\pr(\neg Q_{v,w}|Q_w)\le 2\pr(\neg Q_{v,w})$ (since $\pr(Q_w)\ge 1-2\eps \ge 1/2)$. In order to bound $\pr(\neg Q_{v,w})$, consider a path $w=w_0-w_1-w_2-w_3-w_4=v$ of length 4 that connects $w$ and $v$. Note that the event $\neg Q_{v,w}$ implies that at least one of the edges $\{w_{i-1}, w_i\}$ $(i \in [4])$ is not ideal (recall that an edge $\{u,v\}$ is ideal with respect to $f$ if $\bar f_{N_u}=\bar f_{N_v}=\{k, k+1, \ldots, k+M\}$ for some integer $k$), whose probability is at most $4\cdot4\eps=16\eps$ (see \eqref{prob.ideal}). Therefore, 
\beq{neg.Qvw.small} \pr(\neg Q_{v,w}) \le 16\eps,\enq
and \eqref{neg.vw} is at most $H(32\eps)$.
\end{subproof}

\begin{claim}\label{claim2}
$H(\bar\bff_{N_v}|\mathbf{1}_{Q_w}, \mathbf{1}_{Q_{v,w}}, \bff_{N_w}) =O(\eps \log M)$.
\end{claim}

\begin{subproof}
Note that
\[H(\bar\bff_{N_v}|\mathbf{1}_{Q_w}, \mathbf{1}_{Q_{v,w}}, \bff_{N_w})=\pr(\neg Q_{w})H(\bar\bff_{N_v}|\neg Q_w, \mathbf{1}_{Q_{v,w}}, \bff_{N_w})+\pr(Q_w)H(\bar\bff_{N_v}| Q_w, \mathbf{1}_{Q_{v,w}}, \bff_{N_w}).\]
The first term of the above is
\[\begin{split}\pr(\neg Q_{w})H(\bar\bff_{N_v}|\neg Q_w, \mathbf{1}_{Q_{v,w}}, \bff_{N_w})&\stackrel{\eqref{neg.Qu.small}}{\le} 2\eps \cdot H(\bar\bff_{N_v}|\neg Q_w, \mathbf{1}_{Q_{v,w}}, \bff_{N_w})\\
&=O(\eps \log M),\end{split}\]
where the last equality is because $H(\bar\bff_{N_v}|\neg Q_w, \mathbf{1}_{Q_{v,w}}, \bff_{N_w})\le H(\bar\bff_{N_v}|\bar\bff_{N_w})=O(\log M)$ similarly to the argument in \eqref{nbd.ent}.

The second term is
\[\begin{split}\pr(Q_w)H(\bar\bff_{N_v}| Q_w,& \mathbf{1}_{Q_{v,w}}, \bff_{N_w}) \\
&= \pr(Q_w \wedge Q_{v,w})H(\bar\bff_{N_v}|Q_w \wedge Q_{v,w}, \bff_{N_w})+\pr(Q_w \wedge \neg Q_{v,w})H(\bar\bff_{N_v}|Q_w \wedge \neg Q_{v,w},\bff_{N_w}).\end{split}\]
Note that the first term of the above is 0, since
\[\begin{split}H(\bar\bff_{N_v}|Q_w \wedge Q_{v,w}, \bff_{N_w})&=\sum_c \pr(\bff_{N_w}=c~|Q_w \wedge Q_{v,w})H(\bar\bff_{N_v}|Q_w \wedge Q_{v,w}, \bff_{N_w}=c)\\
&=\sum_c\pr(\bff_{N_w}=c~|Q_w \wedge Q_{v,w})\cdot 0=0.\end{split}\]
The next term is
\[\begin{split}\pr(Q_w \wedge \neg Q_{v,w})H(\bar\bff_{N_v}|Q_w \wedge \neg Q_{v,w},\bff_{N_w}) &\le \pr(\neg Q_{v,w})H(\bar\bff_{N_v}|Q_w \wedge \neg Q_{v,w},\bff_{N_w})\\
& \stackrel{\eqref{neg.Qvw.small}}{\le} 16\eps\cdot H(\bar\bff_{N_v}|Q_w \wedge \neg Q_{v,w},\bff_{N_w})\\
&=O(\eps \log M),\end{split}\]
where the last equality, again, follows from  $H(\bar \bff_{N_v}|Q_w \wedge \neg Q_{v,w}, \bff_{N_w}) \le H(\bar\bff_{N_v}|\bar\bff_{N_w})=O(\log M)$ similarly to the argument in \eqref{nbd.ent}.
\end{subproof}

By combining \Cref{claim1} and \Cref{claim2}, and using the fact that $H(O(\eps))=O(\eps\log(1/\eps))$ for small $\eps$, we may bound the left-side of \eqref{ent4} by
\[H(\bar\bff_{N_v}|\bff(x):x \prec N_v)= O(\eps\log (1/\eps)+\eps \log M).\]
Plugging this back into \eqref{ent3} and simplifying, we obtain that
\[\frac{\eps N}{M^2}=O(\log M+\eps n \log (1/\eps)+\eps n\log M).\]
Since we are assuming $k>\gamma M^2\log(Mn)$ (and $N=nk$), the $O(\eps n\log M)$ term on the right-side is absorbed by the left-side, so the above can be rewritten as
\beq{eps.final}\frac{\eps nk}{M^2}=O(\log M+\eps n \log (1/\eps)).\enq
We argue that $\eps n \log(1/\eps)$ cannot be significantly larger than $\log M$; more precisely, we must have 
\beq{musthave} O(\log M+\eps n \log(1/\eps))=O(\log M).\enq
To see this, assume otherwise for the sake of contradiction. Then, \eqref{eps.final} implies that $k/M^2=O(\log(1/\eps)),$ which in turn implies that (again using $k>\gamma M^2\log(Mn)$) $\eps<(Mn)^{-\gamma'}$ for some large constant $\gamma'$. Thus we have $\eps n\log(1/\eps) \ll \log M$, which is a contradiction.

Now, combining \eqref{eps.final} and \eqref{musthave}, we conclude that
\[\eps=O\left(\frac{\log M}{\gamma n\log(Mn)}\right).\qedhere\]
\end{proof}
    
\section*{Acknowledgement}

This project was conducted as part of the 2024 NYC Discrete Math REU, funded by NSF grant DMS-2349366 and Jane Street. JP is supported by NSF grant DMS-2324978 and a Sloan Fellowship. 

\bibliographystyle{plain}
\bibliography{bibliography}

\end{document}